\documentclass{amsart}
\usepackage{amsmath, amssymb,epic,graphicx,mathrsfs,enumerate}
\usepackage[all]{xy}

\usepackage{amsthm}
\usepackage{amssymb}
\usepackage{latexsym}
\usepackage{longtable}
\usepackage{epsfig}
\usepackage{amsmath}
\usepackage{hhline}

 \DeclareMathOperator{\perm}{Sym}

\DeclareMathOperator{\B}{B}

\DeclareMathOperator{\ran}{rank} \DeclareMathOperator{\frat}{Frat}
\DeclareMathOperator{\ssl}{SL} 

  \DeclareMathOperator{\diam}{diam}

\DeclareMathOperator{\GL}{GL}

\DeclareMathOperator{\End}{End} 
\DeclareMathOperator{\der}{Der}

\newcommand{\C}{\mathbb C}

\renewcommand{\emptyset}{\varnothing}

\newtheorem{thm}{Theorem}%[section]
\newtheorem{cor}[thm]{Corollary}
 \newtheorem{lemma}[thm]{Lemma}
\newtheorem{prop}[thm]{Proposition} \newtheorem{rem}[thm]{Remark}

\numberwithin{equation}{section}

\renewcommand{\footnote}{\endnote}
\newcommand{\ignore}[1]{}\makeglossary

\begin{document}
	\bibliographystyle{amsplain}
%	\keywords{groups generation; waiting time; Sylow subgroups; permutations groups}
	\title[Generating graph]{The diameter of the generating graph\\ of a finite soluble group}

	\author{Andrea Lucchini}
%	\address{Andrea Lucchini\\ Universit\`a degli Studi di Padova\\  Dipartimento di Matematica \lq\lq Tullio Levi-Civita\rq\rq\\ Via Trieste 63, 35121 Padova, Italy\\email: lucchini@math.unipd.it}
%	\thanks{Partially supported by Universit\`a di Padova (Progetto di Ricerca di Ateneo: \lq\lq Invariable generation of groups\rq\rq).}

	\begin{abstract} Let $G$ be a finite 2-generated soluble group and suppose that
	$\langle a_1,b_1\rangle=\langle a_2,b_2\rangle=G$. If either $G^\prime$ is of odd order or $G^\prime$ is nilpotent, then there exists $b \in G$ with  $\langle a_1,b\rangle=\langle a_2,b\rangle=G.$ We construct a soluble 2-generated group $G$ of order $2^{10}\cdot 3^2$ for which the previous result does not hold. However a weaker result is true for every finite soluble group: if 	$\langle a_1,b_1\rangle=\langle a_2,b_2\rangle=G$, then there exist $c_1, c_2$ such that $\langle a_1, c_1\rangle = \langle c_1, c_2\rangle =
	\langle c_2, a_2\rangle=G.$
	\end{abstract}
	\maketitle
	
	\section{Introduction}
	
Let $G$ be a finite group. The generating graph for $G,$ written $\Gamma(G),$ is the graph where the vertices are the nonidentity elements of $G$ and there is an edge between $g_1$ and $g_2$ if $G$ is generated by $g_1$ and $g_2.$ If $G$ is not 2-generated, then there will be no edges in this graph. Thus, it is natural to assume that $G$ is 2-generated when looking at this graph. There could be many isolated vertices in this graph. For example, all of the elements in the Frattini subgroup will be isolated vertices. We can also find isolated vertices outside the Frattini subgroup (for example the nontrivial elements of the Klein subgroup are isolated vertices in  $\Gamma(\perm(4)).$ Let $\Delta(G)$ be the subgraph of $\Gamma(G)$ that is induced by all of the vertices that are not isolated. In \cite{CLis} it is  proved that if $G$ is a 2-generated soluble group, then $\Delta(G)$ is connected. In this paper we investigate the diameter $\diam(\Delta(G))$ of this graph.  

\begin{thm}\label{diquattro}
If $G$ is a 2-generated finite soluble group, then  $\Delta(G)$ is connected and $\diam(\Delta(G)) \leq 3.$
\end{thm}

The situation is completely different if the solubility assumption is dropped. It is an open problem whether or not  $\Delta(G)$ is connected, but even when $\Delta(G)$ is connected, its diameter
 can be arbitrarily large. For example if $G$ is the largest 2-generated direct power of $\ssl(2,2^p)$ and $p$ is a sufficiently large odd prime, then $\Delta(G)$ is connected but $\diam(\Delta(G))\geq 2^{p-2}-1$
 (see \cite[Theorem 5.4]{CL2}).

\

For soluble groups, the bound $\diam(\Delta(G))\leq 3$ given in Theorem  \ref{diquattro} is best possible. In Section \ref{esempio} we  construct a soluble 2-generated group $G$ of order $2^{10}\cdot 3^2$ with $\diam(\Delta(G))=3.$  However we prove that 
$\diam(\Delta(G))\leq 2$ in some relevant cases.

\begin{thm}\label{main2} Suppose that a finite 2-generated soluble group $G$ has property that $|\End_G(V)|>2$ for every nontrivial irreducible $G$-module which is $G$-isomorphic to a complemented chief factor of $G$.
	Then  $\diam(\Delta(G))\leq 2$, i.e. if $\langle a_1,b_1\rangle=\langle a_2,b_2\rangle=G,$ then there exists $b \in G$ with  $\langle a_1,b\rangle=\langle a_2,b\rangle=G.$
\end{thm}

	\begin{cor}\label{coruno} Let $G$ be a 2-generated finite group. If the derived subgroup of $G$  has odd order, then  $\diam(\Delta(G))\leq 2.$
	\end{cor}
	
		\begin{cor}\label{cordue} Let $G$ be a 2-generated finite group. If the derived subgroup of $G$  is nilpotent, then  $\diam(\Delta(G))\leq 2.$ 
		\end{cor}

\section{Proof of Theorem \ref{main2}}

We could prove Theorem \ref{main2} with the same approach that will be used in the proof of Theorem \ref{diquattro}. However we prefer to give in this particular case an easier and shorter proof. Before doing that,
 we briefly recall some necessary definitions and results. Given a subset
$X$ of 
a finite group $G,$ we will denote by $d_X(G)$ the smallest cardinality
of a set of elements of $G$ generating $G$ together with the elements 
of $X.$ The following 
generalizes
a result originally
obtained by W. Gasch\"utz \cite{Ga} for $X=\emptyset.$ 

\begin{lemma}[\cite{CLis} Lemma 6] \label{modg} Let $X$ be a subset of $G$ and $N$ a normal subgroup of $G$ and suppose that
	$\langle g_1,\dots,g_k, X\rangle N=G.$ 
	%\comment{CMRD: I have swapped $k$ and $r$ in Lemma 2.11 and proof of
	%Thm 2.10 to
	%  match what we used earlier}
	If $k\geq d_X(G),$ then there exist
	$n_1,\dots,n_k\in N$ so that $\langle g_1n_1,\dots,g_kn_k,X\rangle=G.$
\end{lemma}

It follows from the proof of \cite[Lemma 6]{CLis} that the number, say  $\phi_{G,N}(X,k),$ of $k$-tuples $(g_1n_1,\dots,g_kn_k)$ generating $G$ with $X$ is 
independent of the choice of $(g_1,\dots,g_k).$ In particular 
$$\phi_{G,N}(X,k)=|N|^kP_{G,N}(X,k)$$
where $P_{G,N}(X,k)$ is the conditional probability that $k$ elements of $G$ generate $G$ with $X,$ given that they generate $G$ with $XN$.

\begin{prop}[\cite{xdir} Proposition 16]\label{proba} If $N$ is a normal subgroup of a finite group $G$ and $k$ is a positive integer, then
$$P_{G,N}(X,k)=\sum_{\substack{X\subseteq H\leq G\\
			HN=G\\}}\frac{\mu(H,G)}{|G:H|^k}.$$
where
$\mu$ is the M\"{o}bius function associated with the subgroup
lattice of $G.$ 	
\end{prop}
	
\begin{cor}\label{dueterzi}
Let $N$ be a minimal normal subgroup of a finite group $G$. Assume that $N$ is abelian and let $q=|\End_G(N)|.$ For every $X\subseteq G$, if $k\geq d_X(G)$ and $P_{G,N}(X,k)\neq 0,$ then  $P_{G,N}(X,k)\geq \frac{q-1}{q}.$
\end{cor}
	
\begin{proof}
We may assume that $N$ is not contained in the Frattini subgroup of $G$ (otherwise $P_{G,N}(X,k)=1$).
In this case, if $H$ is a proper supplement of $N$ in $G,$ then $H$ is a maximal subgroup of $G$ and complements $N$. Therefore $\mu(H,G)=-1$ and $|G:H|=|N|$. It follows from Proposition \ref{proba} that
$$P_{G,N}(X,k)=1-\frac{c}{|N|^k},$$
where $c$ is the number of complements of $N$ in $G$ containing $X.$ If $c=0,$ then $P_{G,N}(X,k)=1.$ Assume $c\neq 0$ and fix a complement $H$ of $N$ in $G$ containing $X.$ Let $\der_X(H,N)$ be the set of derivations $\delta$ from $H$ to $N$ with the property that $x^\delta=1$ for every $x \in X.$ The complements of $N$ in 
$G$ containing $X$ are precisely the subgroups of $G$ of the kind
$H_\delta=\{hh^\delta\mid h \in H\}$ with $\delta \in \der_X(H,N)$, hence
$$P_{G,N}(X,k)=1-\frac{|\der_X(H,N)|}{|N|^k}.$$
Now let $\mathbb F_q=\End_G(N)$. Both $N$ and $\der_X(H,N)$ can be viewed as vector spaces over $\mathbb F_q.$ Let $$n=\dim_{\mathbb F_q}N,\quad a=\dim_{\mathbb F_q}\der_X(H,N).$$ We have
$$P_{G,N}(X,k)=1-\frac{q^a}{q^{nk}}.$$
Since $P_{G,N}(X,k)\neq 0,$ we have $a<kq$ and
$$P_{G,N}(X,k)=1-\frac{q^a}{q^{nk}}\geq 1-\frac{1}{q}=\frac{q-1}{q}.\qedhere$$
\end{proof}

\begin{proof}[Proof of Theorem \ref{main2}]
	We prove the theorem by induction on the order of $G.$ We may assume that $G$ is not cyclic and that the Frattini subgroup of $G$ is trivial. We distinguish two cases:
	
\noindent a) All the minimal normal subgroups of $G$ have order 2. In this case $G$ is an elementary abelian group of order 4 and $a_1$ and $a_2$ are nontrivial elements of $G.$ If $b\notin \{1,a_1,a_2\},$ then
	$\langle a_1,b\rangle=\langle a_2,b\rangle=G.$
	
\noindent b) $G$ contains a minimal normal subgroup $N$ with $|N|\geq 3.$	By assumption $q=|\End_G(N)|\geq 3.$ Assume $\langle a_1, b_1\rangle=\langle a_2, b_2\rangle=G.$ By induction there exists $g\in G$ such that  $\langle a_1, g\rangle N=\langle a_2, g \rangle N=G.$ For $i\in\{1,2\},$ let
$$\Omega_i=\{n\in N \mid \langle a_i, gn\rangle=G\}.$$
Since $\langle a_i, b_i\rangle=G,$ we have $d_{\{a_i\}}(G)\leq 1,$ hence, by Lemma \ref{modg}, $P_{G,N}(\{a_i\},1)\neq 0$ and consequently we deduce from Corollary \ref{dueterzi} that
$$|\Omega_i|=|N|P_{G,N}(\{a_i\},1)\geq |N|\frac{q-1}{q}\geq \frac{2|N|}{3}.$$
But then $\Omega_1\cap \Omega_2\neq \emptyset.$ Let $b=gn$ with 
$n\in \Omega_1\cap \Omega_2.$ Then $G=\langle a_1, b\rangle=\langle a_2, b\rangle.$
	\end{proof}
	
\begin{proof}[Proof of Corollary \ref{coruno}] Let $G^\prime$ be the derived subgroup of $G.$
	If $|G^\prime|$ is odd, then $G^\prime$ is soluble by the Feit-Thompson Theorem, and consequently
	$G$ is also soluble. Moreover if $X$ and $Y$ are normal subgroups of $G$ such that $A=X/Y$ is a nontrivial irreducible $G$-module then  $|A|$ is a power of a prime divisor $p$ of $|G^\prime|$ and $F=\End_G(A)$ is a finite field of characteristic $p$.
	Hence $|F|\geq p\geq 3$ and we may apply Theorem \ref{main2}.
\end{proof}

\begin{proof}[Proof of Corollary \ref{cordue}]
We may assume $\frat (G)=1.$ This means that $G=M\rtimes H$ where $H$ is abelian and $M=V_1\times \cdots \times V_u$ is the direct product
	of $u$ irreducible non trivial $H$-modules $V_1,\dots,V_u.$ Let $F_i=\End_H(V_i)=\End_G(V_i)$: for each $i\in \{1,\dots,u\},$ $V_i$ is an absolutely irreducible $F_iH$-module
	so $\dim_{F_i}V_i=1.$ Now assume that $A$ is a nontrivial irreducible $G$-module $G$-isomorphic to a complemented chief factor of $G$:
	it must be $A\cong_G V_i$ for some $i$, so $|\End_G(A)|=|F_i|=|V_i|=|A|.$ It cannot be $|A|=2,$ otherwise $A$ would be a trivial $G$-module.
	Again we may apply Theorem \ref{main2}.
\end{proof}

\begin{rem}
	{\rm{Let $N$ be a noncentral and complemented minimal normal subgroup of a 2-generated soluble group $G$ and assume that $\mathbb F_2=\End_G(N).$ It follows from \cite[Lemma 18]{xdir} and \cite[Lemma 18]{bias} that if $P_{G,N}(X,k)\leq 1/2,$ then there are $\dim_{\mathbb F_2}N$ different complemented factors $G$-isomorphic to $N$ in every chief series of $G.$ This means that, with the same arguments used in the proof of Theorem \ref{main2}, a little bit stronger result can be proved:}} Suppose that a finite 2-generated soluble group $G$ has the following property: if $V$ is
			a nontrivial irreducible $G$-module with $\End_G(V)=\mathbb F_2,$ then a chief series of $G$ does not contain $\dim_{\mathbb F_2}V$ different complemented factors $G$-isomorphic to $V.$
				Then  $\diam(\Delta(G))\leq 2$.
\end{rem}

\section{A finite soluble group $G$ with $\diam(\Delta(G))>2$}\label{esempio}

Let first recall some results that we will be applied in the discussion of our example. Let $G$ be a finite soluble group, and let $\mathcal V_G$ be a set
of representatives for the irreducible $G$-groups that are
$G$-isomorphic to a complemented chief factor of $G$. For $V \in
\mathcal V_G$ let $R_G(V)$ be the smallest normal subgroup contained
in $C_G(V)$ with the property that $C_G(V)/R_G(V)$ is
$G$-isomorphic to a direct product of copies of $V$ and it has a
complement in $G/R_G(V)$. The factor group $C_G(V)/R_G(V)$ is
called the $V$-crown of $G$. The non-negative integer
$\delta_G(V)$ defined by $C_G(V)/R_G(V)\cong_G V^{\delta_G(V)}$ is
called the $V$-rank of $G$ and it coincides with the number of
complemented factors in any chief series of $G$ that are
$G$-isomorphic to $V$. If $\delta_G(V) \neq 0$, then the $V$-crown
is the socle of $G/R_G(V)$. The notion of crown was introduced by
Gasch\"utz in \cite{Ga2}. We have (see for example \cite[Proposition 2.4]{LM2}):

\begin{prop}\label{prouno}%\cite[Proposition 2.4]{LM2}
	\label{lemma} Let $G$ and $\mathcal V_G$ be as above. Let $x_{1},
	\ldots , x_{u}$ be elements of $G$ such that $\langle
	x_1,\dots,x_u,R_G(V)\rangle=G$ for any $V \in \mathcal V_G$. Then
	$\langle x_1,\dots,x_u\rangle=G$.
\end{prop}

Now let $V$ be a finite dimensional vector space over a finite field
of prime order. Let $K$ be a $d$-generated linear soluble group acting
irreducibly and faithfully on $V$ and fix a generating $d$-tuple $(k_1,\dots,k_d)$ of $K.$
%Suppose that $H$ can be generated by $d$ elements.
For a positive integer $u$ we consider
the semidirect product $G_u = V^u \rtimes K$ where $K$ acts in the
same way on each of the $u$ direct factors. Put $F =
\mathrm{End}_{K}(V)$. Let $n$ be the dimension of $V$ over $F$. We may identify $K =
\langle k_1, \dots, k_d \rangle$ with a subgroup of the general linear group $\GL(n,F)$. In this
identification $k_i$ becomes an $n\times n$ matrix $X_i$ with coefficients in $F$; denote by $A_i$ the matrix
$I_n-X_i.$ Let $w_i=(v_{i,1},\dots,v_{i,u})\in V^u.$
Then every $v_{i,j}$ can be
viewed as a $1 \times n$ matrix. Denote the $u \times n$ matrix
with rows $v_{i,1},\dots,v_{i,u}$  by $B_i$. The following result is proved in
\cite[Section 4]{CL4}.

\begin{prop}\label{richiami} The group $G_u=V^u\rtimes K$ can be generated by $d$ elements if and only if $u\leq n(d-1).$ Moreover
	\begin{enumerate}
		\item $\mathrm{rank} \begin{pmatrix}A_1&\dots&A_d\end{pmatrix}=n.$ \item
		$\langle k_1w_1,\dots,k_dw_d \rangle=V^u \rtimes K$ if and only if
		$\mathrm{rank} \begin{pmatrix}A_1&\cdots&A_d\\
		B_1&\cdots&B_d\end{pmatrix} = n+u.$
	\end{enumerate}
	\end{prop}
In this section	we will use in particular the following corollary of the previous proposition:
	\begin{cor}\label{coro}
Let $V=\mathbb F_2 \times \mathbb F_2,$ where $\mathbb F_2$ is the field with 2 elements and let $\Gamma=\GL(2,2)\ltimes V^2.$ Assume that $\langle k_1, k_2 \rangle =\GL(2,2)$ and let $\gamma_1=k_1(v_1,v_2),$ $\gamma_2=k_2(v_3,v_4)$ in $\Gamma.$ We have that $\Gamma=\langle \gamma_1, \gamma_2\rangle$ if and only if $$\begin{pmatrix}1-k_1&1-k_2\\v_1&v_3\\v_2&v_4\end{pmatrix}\neq 0.$$
	\end{cor}

Now we are ready to start the construction of a finite 2-generated  soluble $G$ with $\diam(\Delta(G))>2.$ Let $H=\GL(2,2)\times \GL(2,2)$ and let $W=V_1\times V_2\times V_3\times V_4$ be the direct product of four 2-dimensional vector spaces over the field $\mathbb F_2$ with two elements. We define an action of $H$ on $W$ by setting
$$(v_1,v_2,v_3,v_4)^{(x,y)}=(v_1^{x},v_2^{x},v_3^{y},v_4^{y})$$
and we consider the semidirect product $$G=H\ltimes W.$$
Let $$\begin{aligned}
N_1:=&C_G(V_3)=C_G(V_4)=\{(k,1)\mid k\in \GL(2,2)\},\\
N_2:=&C_G(V_1)=C_G(V_2)=\{(1,k)\mid k\in \GL(2,2)\}.\\
\end{aligned}$$
A set of representatives for the $G$-isomorphism classes of
the complemented chief factor of $G$ contains precisely 5 elements:
\begin{itemize}
	\item $Z,$ a central $G$-module of order 2, with $R_G(Z)=G^\prime=\ssl(2,2)^2\ltimes W.$
	\item $U_1,$ a non central $G$-module of order
	3, with $R_G(U_1)=N_2\ltimes W.$
		\item $U_2,$ a non central $G$-module of order
		3, with $R_G(U_2)=N_1\ltimes W.$
	\item $V_1$, with $R_G(V_1)=V_3\times V_4 \times N_2.$
	\item $V_3$, with $R_G(V_3)=V_1\times V_2 \times N_1.$
\end{itemize}
Let $$(x_1,y_1)(v_{11},v_{12},v_{13},v_{14})=g_1, \quad (x_2,y_2)(v_{21},v_{22},v_{23},v_{24})=g_2.$$
We want to apply  Proposition \ref{prouno} to check whether $\langle g_1,g_2 \rangle=G$. The three conditions
$$\langle g_1,g_2\rangle R_G(Z)=G, \langle g_1,g_2\rangle R_G(U_1)=G,  \langle g_1,g_2\rangle R_G(U_2)=G$$ are equivalent to
$\langle g_1,g_2\rangle W=G$, i.e. to $\langle (x_1,y_1), (x_2,y_2)\rangle=H.$ Moreover 
$$\langle g_1,g_2\rangle R_G(V_1)=G \text { if and only if } 
\langle x_1(v_{11},v_{12}), x_2(v_{21},v_{22})\rangle=(V_1\times V_2)\rtimes \GL(2,2),$$ 
$$\langle g_1,g_2\rangle R_G(V_3)=G \text { if and only if } 
\langle y_1(v_{31},v_{32}), y_2(v_{41},v_{42})\rangle=(V_3\times V_4)\rtimes \GL(2,2).$$ 
Applying Corollary \ref{coro} we conclude that
$$\langle g_1,g_2 \rangle=G$$ if and only if the following conditions are satisfied:
$$\begin{aligned}(1)&\quad \quad \quad \quad \quad
\langle (x_1,y_1), (x_2,y_2)\rangle=H=\GL(2,2)\times \GL(2,2),\\
(2)&\quad \quad \quad \quad \quad\det\begin{pmatrix}1-x_1&1-x_2\\v_{11}&v_{21}\\v_{12}&v_{22}
\end{pmatrix}\neq 0,\\
(3)&\quad \quad \quad \quad \quad\det\begin{pmatrix}1-y_1&1-y_2\\v_{13}&v_{23}\\v_{14}&v_{24}
\end{pmatrix}\neq 0.
\end{aligned}$$
Consider the following elements of $\GL(2,2)$:
 $$x:=\begin{pmatrix}1&0\\1&1\end{pmatrix},\quad
y:=\begin{pmatrix}1&1\\1&0\end{pmatrix}, \quad z:=\begin{pmatrix}1&1\\0&1\end{pmatrix},$$
and the following elements of $\mathbb F_2^2$:
 $$0=(0,0),\quad e_1=(1,0),\quad e_2=(0,1).$$
Let
$$\begin{aligned}a_1:=&(x,x)(0,e_2,0,e_2),\quad a_2:=(x,x)(e_1,e_2,e_1,e_2)\\
b_1:=&(y,z)(e_1,0,e_1,0),\quad b_2:=(y,z)(0,0,e_1,0).
\end{aligned}$$
It can be easily checked that
$$\langle (x,x), (y,z) \rangle = H.$$
Moreover
$$\det\begin{pmatrix}1-x&1-y\\0&e_1\\e_2&0\end{pmatrix}=\det\begin{pmatrix}0&0&0&1\\1&0&1&1\\0&0&1&0\\0&1&0&0\end{pmatrix}=1,$$
$$\det\begin{pmatrix}1-x&1-z\\0&e_1\\e_2&0\end{pmatrix}=\det\begin{pmatrix}0&0&0&1\\1&0&0&0\\0&0&1&0\\0&1&0&0\end{pmatrix}=1,$$
$$\det\begin{pmatrix}1-x&1-y\\e_1&0\\e_2&0\end{pmatrix}=\det\begin{pmatrix}0&0&0&1\\1&0&1&1\\1&0&0&0\\0&1&0&0\end{pmatrix}=1,$$
$$\det\begin{pmatrix}1-x&1-z\\e_1&e_1\\e_2&0\end{pmatrix}=\det\begin{pmatrix}0&0&0&1\\1&0&0&0\\1&0&1&0\\0&1&0&0\end{pmatrix}=1,$$
so either $a_1,b_1$ as $a_2,b_2$ satisfy the three conditions (1), (2) (3) and therefore $$\langle a_1, b_1\rangle=\langle a_2, b_2\rangle=G.$$
Now we want to prove that there is no $b\in G$ with  $\langle a_1, b\rangle=\langle a_2, b\rangle=G.$ Let $b=(h_1,h_2)(v_1,v_2,v_3,v_4)
$, and assume by contradiction that $\langle a_1, b\rangle=\langle a_2, b\rangle=G.$ We must have in particular that condition (1) holds, i.e. $\langle (x,x), (h_1,h_2)\rangle=H.$ Since $(x,x)$ has order 2 and $H$ cannot be generated by two involutions (otherwise it would be a dihedral group) at least one of the two elements $h_1, h_2$ must have order 3: it is not restrictive to assume $h_1=y.$ Let $v_1=(\alpha,\beta),$ $v_2=(\gamma,\delta).$ Conditions (2) and (3) must be satisfied, hence we must have
$$\det\begin{pmatrix}1-x&1-y\\0&v_1\\e_2&v_2\end{pmatrix}=
\det\begin{pmatrix}1-x&1-y\\e_1&v_1\\e_2&v_2\end{pmatrix}=1.$$
However 
$$\det\begin{pmatrix}1-x&1-y\\0&v_1\\e_2&v_2\end{pmatrix}=
\det\begin{pmatrix}0&0&0&1\\1&0&1&1\\0&0&\alpha&\beta\\0&1&\gamma&\delta\end{pmatrix}=\alpha,$$
$$\quad \quad \det\begin{pmatrix}1-x&1-y\\e_1&v_1\\e_2&v_2\end{pmatrix}=
\det\begin{pmatrix}0&0&0&1\\1&0&1&1\\1&0&\alpha&\beta\\0&1&\gamma&\delta\end{pmatrix}=\alpha+1.$$
However, since $\alpha\in \mathbb F_2$ either $\alpha=0$ or $\alpha+1=0,$
so there is no $b\in G$ with  $\langle a_1, b\rangle=\langle a_2, b\rangle=G.$

\section{A problem in linear algebra}

Before to prove Theorem \ref{diquattro}, we need to collect a series of results in linear algebra. Denote by $M_{r\times s}(F)$ the set of the $r\times s$ matrices with coefficients over the field $F.$

\begin{lemma}\label{comcom}{\cite[Lemma 3]{CLis}}
	Let $V$ be a finite dimensional vector space over the field $F$.
	If $W_1$ and $W_2$ are subspaces of $V$ with $\dim W_1=\dim W_2$,
	then $V$ contains a subspace $U$ such that $V=W_1\oplus U=W_2\oplus U.$
\end{lemma}

\begin{lemma}\label{somma}
	Let $v_1,\dots,v_n,w_1,\dots w_n \in F^n$, where $F$ is a finite field and either $|F|>2$ or $n\neq 1.$ There exist
	$z_1,\dots,z_n\in F^n$ so that the two sequences
	$$v_1+z_1,\dots,v_n+z_n,$$
	$$w_1+z_1,\dots,w_n+z_n$$ are
	both basis of $F^n.$
\end{lemma}

\begin{proof}
Equivalently, we want to prove that for every pair of matrices $A, B \in
M_{n\times n}(F),$ there exists $C\in M_{n\times n}(F),$ such that $\det(A+C)\neq 0$ and $\det(B+C)\neq 0.$
Since either $|F|>2$ or $n\neq 1,$ every element of $M_{n\times n}(F)$ can be expressed as the sum of two units  \cite{dz}. In particular $A-B=U-V$ with $U, V \in \GL(n,F).$ We may take $C=U-A=B-V.$
\end{proof}

\begin{lemma}\label{prs}Let $F$ be a finite field and assume $r\leq n.$ Given $R\in  M_{r\times n}(F)$ and $S\in M_{r\times r}(F)$ consider the matrix 
$\begin{pmatrix}R&S\end{pmatrix}\in M_{r\times (n+r)}$. Assume
$\ran \begin{pmatrix}R&S\end{pmatrix}=r$ and let $\pi_{R,S}$ be the probability that a matrix $Z \in M_{r\times n}(F)$ satisfies the condition $\ran(R+SZ)=r.$ Then $$\pi_{R,S} > 1-\frac{q^r}{q^n(q-1)}.$$
\end{lemma}

\begin{proof}
There exist $m\leq r,$  $X \in \GL(r,F)$ and $Y \in \GL(r,F)$ such that
 $$XSY=\begin{pmatrix}I_m&0\\0&0\end{pmatrix},$$  where $I_m$ is the identity element in $M_{m\times m}(F).$
Since
 $$r=\ran \begin{pmatrix}R&S\end{pmatrix}=\ran\left( X\begin{pmatrix}R&S\end{pmatrix}\begin{pmatrix}I_n&0\\0&Y\end{pmatrix}\right)=\ran\begin{pmatrix}XR&XSY\end{pmatrix}$$
and   $$\begin{aligned}\ran(R+SZ)&=\ran(X(R+SZ))=\ran(XR+XSZ)\\
&= \ran(XR+XSY(Y^{-1}Z)),\end{aligned}$$ 
it is not restrictive (replacing $R$ by $XR$, $S$ by $XSY$ and $Z$ by $Y^{-1}Z$) to assume 
$$S=\begin{pmatrix}I_m&0\\0&0\end{pmatrix}.$$	
 Denote by $v_1,\dots,v_r$ the rows of $R$ and by $z_1,\dots,z_r$ the rows of $Z.$ The fact that the rows of $(R\ S)$ are linearly independent implies that $v_{m+1},\dots,v_r$ are linearly independent vectors of $F^n.$ The condition $\ran(R+SZ)=r$ is equivalent to ask that
 $$v_1+z_1,\dots,v_m+z_m,v_{m+1},\dots,v_r$$ are linearly independent.
The probability that $z_1,\dots,z_m$ satisfy this condition is 
 $$\left(1-\frac{q^{r-m}}{q^n}\right)\left(1-\frac{q^{r-m+1}}{q^n}\right)\cdots \left(1-\frac{q^{r-m+(m-1)}}{q^n}\right).$$
	Hence\
	$$\begin{aligned}\pi_{R,S}&=\left(1-\frac{q^{r-m}}{q^n}\right)\left(1-\frac{q^{r-m+1}}{q^n}\right)\cdots \left(1-\frac{q^{r-m+(m-1)}}{q^n}\right)\\&\geq 1-\frac{q^{r-m}(1+q+\dots+q^{m-1})}{q^n}\\&=1-\frac{q^{r-m}(q^m-1)}{q^n(q-1)}>1-\frac{q^r}{q^n(q-1)}.\qedhere\end{aligned}$$   
\end{proof}
 \begin{lemma}\label{que}
Assume that $F$ is a finite field and that $A$, $B_1,$ $B_2,$ $D_1$ and $D_2$ are elements of $M_{n\times n}(F)$ with the property that
$$\ran \begin{pmatrix}A&B_1\end{pmatrix}=\ran \begin{pmatrix}A&B_2\end{pmatrix}=n,$$
$$\ran \begin{pmatrix}B_1\\D_1\end{pmatrix}=\ran \begin{pmatrix}B_2\\D_2\end{pmatrix}=n.$$ 
Moreover assume that  at least one of the following conditions holds:
\begin{enumerate}
\item $|F|>2$; \item $\det A=0$;
	\item $n\geq 2$ and $(\det B_1, \det B_2)\neq (0,0).$
\end{enumerate}Then there exists $C\in M_{n\times n}(F)$ such that
 $$\det\begin{pmatrix}A&B_1\\C&D_1\end{pmatrix}\neq 0 \quad \text { and } \quad \det\begin{pmatrix}A&B_2\\C&D_2\end{pmatrix}\neq 0.$$
 \end{lemma}
\begin{proof} Let $r=\ran(A).$ There exist $X,Y \in \GL(n,F)$ such that
 $$XAY=\begin{pmatrix}I_r&0\\0&0\end{pmatrix}$$
 where $I_r$ is the identity element in $M_{r\times r}(F).$
 Let $B_{11}, B_{21}\in M_{r\times n}(F)$ and $B_{12}, B_{22}\in M_{(n-r)\times n}(F)$
 such that
 $$XB_1=\begin{pmatrix}B_{11}\\B_{12}\end{pmatrix},\quad XB_2=\begin{pmatrix}B_{21}\\B_{22}\end{pmatrix}.$$
For $i\in\{1,2\},$ since
 $$n=\ran \begin{pmatrix}A & B_i\end{pmatrix}=
 \ran \left(X\begin{pmatrix}A&B_i\end{pmatrix}\begin{pmatrix}Y&0\\0&I_n\end{pmatrix}\right)=\ran \begin{pmatrix}I_r&0&B_{i1}\\0&0&B_{i2}\end{pmatrix},$$
 it must be $\ran (B_{i2})=n-r$.
In particular there exists $Z_i\in \GL(n,F)$ such that
$$XB_iZ_i=\begin{pmatrix}B_{i1}\\B_{i2}\end{pmatrix}Z_i=\begin{pmatrix}B_{i1}^*&B_{i2}^*\\0&I_{n-r}\end{pmatrix}$$
with $B_{i1}^* \in M_{r\times r}(F),$ $B_{i2}^* \in M_{r\times (n-r)}(F).$
Notice that
$$\begin{aligned}\det\begin{pmatrix}XAY&XB_iZ_i\\CY&D_iZ_i\end{pmatrix}&=
\det\left(\begin{pmatrix}X&0\\0&I_n\end{pmatrix}\begin{pmatrix}A&B_i\\C&D_i\end{pmatrix}
\begin{pmatrix}Y&0\\0&Z_i\end{pmatrix}
\right)\\&=
\det(X)\det(Y)\det(Z_i)\det\begin{pmatrix}A&B_i\\C&D_i\end{pmatrix}.
\end{aligned}$$
This means that it is not restrictive to assume
$$A=\begin{pmatrix}I_r&0\\0&0\end{pmatrix},\quad B_i=\begin{pmatrix}B_{i1}^*&B_{i2}^*\\0&I_{n-r}\end{pmatrix}$$
with $B_{i1}^* \in M_{r\times r}(F),$ $B_{i2}^* \in M_{r\times (n-r)}(F).$
Let $C_1, D_{i1}\in M_{n\times r}(F)$ and $C_2,D_{i2}\in M_{n\times (n-r)}(F)$ such that $\begin{pmatrix}C_1&C_2\end{pmatrix}=C$ and
$\begin{pmatrix}D_{i1}&D_{i2}\end{pmatrix}=D.$ Notice that
	$$\quad\quad\quad \begin{aligned} \det\begin{pmatrix}A&B_i\\C&D_{i}\end{pmatrix}=&
	\det\begin{pmatrix}I_r&0&B_{i1}^*&B_{i2}^*\\0&0&0&I_{n-r}\\C_1&C_{2}&D_{i1}&D_{i2}\end{pmatrix}\\&=
	(-1)^n\det\begin{pmatrix}I_r&0&B_{i1}^*\\C_1&C_{2}&D_{i1}\end{pmatrix}\\&=(-1)^n\det\left(\begin{pmatrix}I_r&0&B_{i1}^*\\C_1&C_{2}&D_{i1}\end{pmatrix}
	\begin{pmatrix}I_r&0&-B_{i1}^*\\0&I_{n-r}&0\\0&0&I_r\end{pmatrix}\right)\\&=(-1)^n\det\begin{pmatrix}I_r&0&0\\C_1&C_{2}&D_{i1}-C_1B_{i1}^*\end{pmatrix}\\
	&=(-1)^n\det\begin{pmatrix}C_{2}&D_{i1}-C_1B_{i1}^*\end{pmatrix}.
	\end{aligned}$$
Assume that we can find $C_1$ such that $$\ran (D_{11}-C_1B_{11}^*) = \ran (D_{21}-C_1B_{21}^*)=r$$ and let $W_1, W_2$ be the subspaces of $F^n$ spanned, respectively, by the columns of the two matrices $D_{11}-C_1B_{11}^*$ and
$D_{21}-C_1B_{21}^*$. By Lemma \ref{comcom}, there exists a subspace $U$ of $F^n$ such that $F^n=W_1\oplus U=W_2\oplus U.$ If $C_2$ is a matrix whose columns are a basis for $U,$ then  
	$$\det\begin{pmatrix}C_{2}&D_{11}-C_1B_{11}^*\end{pmatrix}\neq 0
	 \text { and } \det\begin{pmatrix}C_{2}&D_{21}-C_1B_{21}^*\neq 0\end{pmatrix}
	 $$ and $C=(C_1 \ C_2)$ is a matrix with the request property.
Set $$R_1=D_{11}^\text{T},\ R_2=D_{21}^\text{T},\ S_1=B_{11}^{*\text{T}},\ S_2=B_{21}^{*\text{T}},\ Z= -C_1^\text{T}.$$ The previous observation implies that a matrix $C$ with the requested properties exists if and only if there exists $Z\in M_{r\times n}(F)$ such that
\begin{equation}\ran(R_1+S_1Z)=\ran(R_2+S_2Z)=r.
\end{equation} Notice that $R_1, R_2\in M_{r\times n}(F)$, $S_1, S_2\in M_{r\times r}(F)$ have the property that $$\ran \begin{pmatrix}R_1&S_1\end{pmatrix}=\ran \begin{pmatrix}R_2&S_2\end{pmatrix}=r.$$
First assume that either $|F|=q>2$ or $r<n:$ by Lemma \ref{prs}, we have
$$\pi_{R_1,S_1}>\frac{1}{2} \quad \text { and }\quad  \pi_{R_2,S_2}>\frac{1}{2}$$ and this is sufficient to ensure that a matrix $Z$ with the requested property exists. Therefore we may assume $r=n$ (i.e. $\det A\neq 0$) and $q=2.$ In this case, we assume also that at least one of the two matrices $B_1$ and $B_2$ is invertible. Let for example $\det B_1\neq 0.$ This implies $\det S_1\neq 0.$  There exist $m\leq n$ and $X,Y \in \GL(n,F)$ such that
$$XS_2Y=\begin{pmatrix}I_m&0\\0&0\end{pmatrix}.$$ Notice that $R_2+S_2Z$ is invertible if and only if $XR_2Y+XS_2YY^{-1}ZY$ is invertible. Moreover  $R_1+S_1Z$ is invertible if and only if $R_1Y+S_1YY^{-1}ZY$ is invertible, if and only if $(S_1Y)^{-1}R_1Y+Y^{-1}ZY$ is invertible. This means that (replacing $R_1$ by $(S_1Y)^{-1}R_1Y$, $R_2$ by $XR_2Y$ and $Z$ by $Y^{-1}ZY$) we may assume
$$S_1=I_n \text { and }
S_2=\begin{pmatrix}I_m&0\\0&0\end{pmatrix}.$$
Let $v_1,\dots,v_n$ be the rows of $R_1,$ $w_1,\dots,w_n$ the rows of $R_2$ and $z_1,\dots,z_n$ the rows of $Z.$ Our request on $Z$ is equivalent to ask that the sequences
$$v_1+z_1,\dots,v_m+z_m,v_{m+1}+z_{m+1},\dots,v_n+z_n,$$
$$w_1+z_1,\dots,w_m+z_m,w_{m+1},\dots,w_n,$$
are both linearly independent. Notice that the condition $\ran (R_2\ S_2)=n$ implies in particular that $w_{m+1},\dots,w_n$ are linearly independent. First assume $m\neq 1$. For $j>m,$ let $z_j=v_j+w_j$ so that
$z_j+v_j=w_j$ and let $W=\langle w_{m+1},\dots,w_n\rangle.$ We then work in the vector space $F^n/W$ of dimension $m$ and our request is that the vectors $v_1+z_1+W,\dots,v_m+z_m+W$ and the vectors $w_1+z_1+W,\dots,w_m+z_m+W$ are linearly independent: Lemma \ref{somma} ensures that this request is fulfilled for a suitable choice of $z_1,\dots,z_m.$ Finally assume $m=1$. As before for  $j>2,$ let $z_j=v_j+w_j$ and let $W=\langle w_{3},\dots,w_n\rangle.$
We want to find $z_1$ and $z_2$ so that the two vectors
$v_1+z_1+W, v_2+z_2+W$ and the two vectors $w_1+z_1+W, w_2+W$ are linearly independent. This is always possible. First choose $z_1$ so that
$\langle w_1+z_1+W\rangle \notin \langle w_2+W\rangle$ and $v_1+z_1\notin W.$ Once $z_1$ has been fixed, choose $z_2$ so that
$\langle v_2+z_2+W\rangle \notin \langle v_1+z_1+W\rangle.$
\end{proof}

\begin{rem}
{\rm{Notice that when $|F|=2$ and $\det A\neq 0,$ we cannot drop the assumption $(\det B_1,\det B_2)\neq (0,0).$ Consider for example
		$$A=\begin{pmatrix}0&1\\1&1\end{pmatrix}, \ B_1=B_2=\begin{pmatrix}0&0\\1&0\end{pmatrix},\ D_1=\begin{pmatrix}0&0\\0&1\end{pmatrix},\ D_2=\begin{pmatrix}1&0\\0&1\end{pmatrix}.$$
		Then, as we noticed at the end of Section \ref{esempio}, there is no $C \in M_{2\times 2}(F)$ with $$\det\begin{pmatrix}A&B_1\\C&D_1\end{pmatrix}\neq 0 \quad \text { and } \quad \det\begin{pmatrix}A&B_2\\C&D_2\end{pmatrix}\neq 0.$$	
This restriction in the statement of Lemma \ref{que} is indeed the reason why we cannot have $\diam(\Delta(G))=2$ for every 2-generated finite soluble group $G.$}}
\end{rem}

\begin{prop} Let $F$ be a finite field and let $n$ be a positive integer. Assume that either $n\geq 2$ or $|F|>2.$ Assume that $A_0$, $A_1,$ $A_2,$ $A_3,$  $B_0$ and $B_3$  are elements of $M_{n\times n}(F)$ with the property that
	$$\ran (A_0 \ A_1)=\ran (A_1 \ A_2)=\ran (A_2\ A_3)$$ and
	$$\ran \begin{pmatrix}A_0\\B_0\end{pmatrix}=\ran \begin{pmatrix}A_3\\B_3\end{pmatrix}=n.$$ Then there exist $B_1, B_2 \in M_{n\times n}(F)$
 such that
	$$\det\begin{pmatrix}A_0&A_1\\B_0&B_1\end{pmatrix}\neq 0,\ \det\begin{pmatrix}A_1&A_2\\B_1&B_2\end{pmatrix}\neq 0,\
\det\begin{pmatrix}A_2&A_3\\B_2&B_3\end{pmatrix}\neq 0. 	
$$
\end{prop}

\begin{proof}
Set $i=1$ if either $|F|>2$ or $|F|=2$ and $\det A_1=0,$ $i=2$ otherwise.
\ 

\noindent a) Assume $i=1.$ First choose $B_2$ so that $$\det\begin{pmatrix}A_2&A_3\\B_2&B_3\end{pmatrix}\neq 0.$$ Then, since either $\det A_1=0$ or $|F|>2,$  by Lemma \ref{que} there exists $B_1$ such that 	$$\det\begin{pmatrix}A_0&A_1\\B_0&B_1\end{pmatrix}\neq 0,\ \det\begin{pmatrix}A_1&A_2\\B_1&B_2\end{pmatrix}\neq 0.$$

\noindent b) Assume $i=2.$ First choose $B_1$ so that 	$$\det\begin{pmatrix}A_0&A_1\\B_0&B_1\end{pmatrix}\neq 0 .$$ Then, since $\det A_1\neq 0$ or $|F|>2,$  by Lemma \ref{que} there exists $B_2$ such that 	$$\det\begin{pmatrix}A_1&A_2\\B_1&B_2\end{pmatrix}\neq 0,\
\det\begin{pmatrix}A_2&A_3\\B_2&B_3\end{pmatrix}\neq 0. \qedhere $$
\end{proof}

\begin{cor}\label{corfour}
Let $K$ be a non-trivial 2-generated linear soluble group acting irreducibly and faithfully on $V$ and consider the semidirect product $G=V^\delta \rtimes K$ with $\delta\leq n=\dim_{\End_G(V)}V.$ Assume that there exists $x_0,x_1,x_2,x_3,$ in $K$ such that
\begin{enumerate}
\item $x_0w_0$ and $x_3w_3$ are non isolated vertices in the generating graph of
$G,$
\item $\langle x_0, x_1\rangle=\langle x_1, x_2\rangle=\langle x_2, x_3\rangle=K.$
\end{enumerate}
Then there exist $w_1,w_2 \in G$ with
$$\langle x_0, x_1w_1\rangle=\langle x_1w_1, x_2w_2\rangle=\langle x_2w_2, x_3\rangle=G.$$
\end{cor}

\begin{proof}
Since $V^\delta\rtimes K$ is an epimorphic image of $V^n\rtimes K$, it suffices to prove the statement in the particular case $G=V^n\times K.$
We may identify $x_0, x_1, x_2, x_3$ with $X_0, X_1, X_2, X_3 \in \GL(n,F)$, where $F=\End_G(V)$ and  $w_0, w_1, w_2, w_3\in V^n$ with four matrices $B_0,B_1,B_2,B_3$ in
$M_{n\times n}(F).$ We now apply Proposition \ref{richiami}. Let
$$A_0=I_n-X_0,\ A_1=I_n-X_1,\ A_2=I_n-X_2,\ A_3=I_n-X_3.$$
Conditions (1) and (2)  implies that 	$$\ran (A_0 \ A_1)=\ran (A_1 \ A_2)=\ran (A_2\ A_3)$$ and
$$\ran \begin{pmatrix}A_0\\B_0\end{pmatrix}=\ran \begin{pmatrix}A_3\\B_3\end{pmatrix}=n.$$ Moreover the statement is equivalent to say that there exist $B_1, B_2 \in M_{n\times n}(F)$
with
$$\det\begin{pmatrix}A_0&A_1\\B_0&B_1\end{pmatrix}\neq 0,\ \det\begin{pmatrix}A_1&A_2\\B_1&B_2\end{pmatrix}\neq 0,\
\det\begin{pmatrix}A_2&A_3\\B_2&B_3\end{pmatrix}\neq 0. 	
$$
The existence of $B_1$ and $B_2$ is ensured by Lemma \ref{que} (notice that the fact that $K$ is a non-trivial subgroup of $\GL(n,F)$ implies that
$n\geq 2$ if $|F|=2$).
\end{proof}

\section{Proof of Theorem \ref{diquattro}}
At the beginning of Section \ref{esempio} we recalled some properties of the crowns of a finite soluble group. In the proof of Theorem  \ref{diquattro}, we will use other two related results.
\begin{lemma}{\cite[Lemma 1.3.6]{classes}}\label{corona}
	Let $G$ be a finite solvable group with trivial Frattini subgroup. There exists
	a crown $C/R$ and a non trivial normal subgroup $U$ of $G$ such that $C=R\times U.$
\end{lemma}

\begin{lemma}{\cite[Proposition 11]{crowns}}\label{sotto} Assume that $G$ is a finite soluble group with trivial Frattini subgroup and let $C, R, U$ as in the statement of Lemma \ref{corona}. If $HU=HR=G,$ then $H=G.$
\end{lemma}

\begin{proof}[Proof of Theorem \ref{diquattro}]
We prove the theorem making induction on the order of $G$. Choose two non-isolated vertices
$x$ and $y$ in the generating graph of $G$. Let $F=\frat(G)$ be the Frattini subgroup of
$G$. Clearly $xF$ and $yF$ are non-isolated vertices of the generating graph of $G/F.$ If
$F\neq 1,$ then by induction there exists a path
$$xF=g_0F,\dots,g_nF=yF$$
in the graph $\Gamma(G/F)$, with $n\leq 3.$ For every $0\leq i \leq n-1,$ we have $G=\langle g_i, g_{i+1}\rangle F
=\langle g_i, g_{i+1}\rangle,$ hence $x=g_0,\dots,g_n=y$ is a path in $\Gamma(G).$ Therefore we may assume $F=1.$ In this case, by Lemma \ref{corona}, there exists a crown $C/R$ of $G$ and
a normal subgroup $U$ of $G$ such that $C=R\times U.$
We have $R=R_G(A)$ where $A$ is an irreducible $G$-module  and $U\cong_G A^\delta$ for $\delta=\delta_G(A).$ 
By induction the graph $\Gamma(G/U)$ contains  a path $xU=g_0U,g_1U,\dots,g_{n-1}U,g_nU=yU$
with $n\leq 3.$ We may assume $n=3:$ indeed if $n=1$ we may consider the path $g_0U, g_1U, g_0U, g_1U$ and if $n=2$ we  may consider the path
$g_0U, g_0g_1U,  g_1U, g_2U.$ So we are assuming
\begin{equation}\label{modu}
\langle x ,g_1 \rangle U=\langle g_1,g_2 \rangle U=\langle g_2,y \rangle U=G.
\end{equation}
We work in the factor group $\bar G=G/R.$ We have $\bar C=C/R=UR/R\cong U\cong A^\delta$
and either $A\cong C_p$ is a trivial $G$-module and $\bar G\cong (C_p)^\delta$  or $\bar G= \bar U \rtimes \bar H \cong A^\delta \rtimes K$
where $K \cong \bar H$ acts in the same say on each of the $\delta$ factors of $A^\delta$ and this action
is faithful and irreducible. Since $\bar G$ is 2-generated, we have $\delta\leq 2$ if $A$ is a trivial $G$-module,  $\delta\leq n:=\dim_{{\End_G(A)}}A$
otherwise. By Theorem \ref{main2} in the first case (we are working in the nilpotent group $A^\delta$) and by Proposition \ref{corfour} in the second case, there exist $u_1,u_2\in U$ with
$\langle \bar x, \bar g_1\bar u_1\rangle=
\langle \bar g_1\bar u_1, \bar g_2\bar u_2\rangle= \langle \bar g_2\bar u_2, \bar y\rangle=\bar G.$ i.e.
\begin{equation}\label{modr}
\langle x,g_1u_1 \rangle R=\langle g_1u_1,g_2u_2 \rangle R=\langle g_2u_2,y \rangle R=G.
\end{equation}
By Lemma \ref{sotto}, from (\ref{modu}) and (\ref{modr}), we deduce $$\langle x, g_1u_1\rangle=\langle g_1u_1, g_2u_2\rangle=\langle g_2u_2, y\rangle=G.\qedhere $$
\end{proof}

\

	\noindent Andrea Lucchini\\ Universit\`a degli Studi di Padova\\  Dipartimento di Matematica \lq\lq Tullio Levi-Civita\rq\rq\\ Via Trieste 63, 35121 Padova, Italy\\email: lucchini@math.unipd.it

\end{document}